\documentclass[runningheads]{llncs}
\usepackage[T1]{fontenc}

\usepackage{graphicx}
\usepackage{tikz}
\usetikzlibrary{positioning,automata}
\usepackage{amsmath}
\usepackage{amssymb}
\usepackage{hyperref}
\usepackage{color}
\usepackage{xypic}[all]
\usepackage[utf8]{inputenc}
\usetikzlibrary{decorations.pathreplacing}
\newcommand{\RSurHom}{\ensuremath{\textsc{SurHom}}}

\newcommand{\NP}{{\sf NP}}

\begin{document}
           
\title{Complexity Framework For Forbidden Subgraphs V: Beyond Simple Graphs}

\titlerunning{Complexity Framework For Forbidden Subgraphs V} 

\author{Tala Eagling-Vose\inst{1}\orcidID{0009-0008-0346-7032} \and
Barnaby Martin\inst{1}\orcidID{0000-0002-4642-8614} \and
Daniel Paulusma\inst{1}\orcidID{0000-0001-5945-9287}
Siani Smith\inst{2}\orcidID{0000-0003-0797-0512}
}

\authorrunning{T. Eagling-Vose, B. Martin, D. Paulusma, S. Smith}

\institute{Durham University, Durham, United Kingdom \email{\{tala.j.eagling-vose,barnaby.d.martin,daniel.paulusma\}@durham.ac.uk} \and
Loughborough University, Loughborough, United Kingdom
\email{siani.smith16@lboro.ac.uk}
}

\maketitle 

\begin{abstract}
We investigate how allowing self‑loops and multi‑edges (i.e. moving from simple graphs to \textit{pseudographs}) impacts the complexity of graph problems on finitely defined monotone graph classes. We do this by extending the C123‑framework beyond simple graphs. We use several testbed problems to illustrate some of the new and varied behaviours introduced by the transition to pseudographs. We show that {\sc Pseudograph Matching Cut} and {\sc Pseudograph $d$-Cut} are C123-problems. Consequently, for every finite set $\mathcal{H}$ of simple graphs, we obtain a complete complexity classification for these problems on $\mathcal{H}$-subgraph‑free graphs. This sharply contrasts with their simple‑graph counterparts, which are not C123‑problems and their complexity remains open even for classes defined by a single forbidden subgraph~$H$. We also show that {\sc Pseudograph Stable Cut} is not a C123‑problem. Nevertheless, for graph classes defined by a single forbidden subgraph, the problem mirrors the C123 classification; this correspondence already fails for sets of size~$2$. This gives the first example of a problem that aligns with C123-problems for a single forbidden subgraph but not for larger sets. Motivated by the interpretation of  {\sc Stable Cut} as a surjective homomorphism to a partially reflexive path, we also study surjective homomorphisms to partially reflexive cycles on $k$ vertices with one irreflexive vertex. These problems share structural similarities with {\sc Pseudograph Stable Cut}. We prove that for $k \geq 6$ , they are NP‑complete on triangle‑free pseudographs.  However, unlike {\sc Pseudograph Stable Cut}, they are polynomial‑time solvable on all graph classes with bounded degree.

\end{abstract}

\section{Introduction}\label{s-intro}
A graph $G$ contains a graph $H$ as a {\it subgraph} if $H$ can be obtained from $G$ by a sequence of vertex deletions and edge deletions (which may be self-loops or multi-edges if $G$ is not simple). For a set of graphs $\mathcal H$, a graph $G$ is {\it $\mathcal H$-subgraph-free} if it contains no graph from $\mathcal H$ as a subgraph. Graph classes closed under vertex and edge deletion are called \emph{monotone}. Every monotone graph class can be characterised by a unique (possibly infinite) set of forbidden subgraphs. If this set is finite, then we say that the resulting graph class is \textit{finitely defined}.

Here, we study the complexity of graph problems on finitely defined monotone graph classes where the input graph is no longer simple but may contain multi-edges and self-loops (\textit{pseudographs}). Our primary focus will be on the two closely related problems {\sc Matching Cut} and {\sc Stable Cut}. The problem {\sc Matching Cut} asks whether a simple graph contains an edge cut whose edges form a matching, whereas {\sc Stable Cut} asks whether a simple graph has a vertex cut that is an independent set. For any simple graph $G$ of minimum degree at least~$2$, it holds that $G$ has a matching cut if and only if the line graph of $G$ has a stable cut.

Previous work on {\sc Matching Cut} on multi-graphs showed that introducing multi-edges may alter the problem’s complexity on certain monotone classes~\cite{MatchingCut}. We show that like {\sc Matching Cut}, the transition from  simple graphs to pseudographs has a significant impact on {\sc Stable Cut}. Through a systematic study of monotone graph classes, we show that this shift to pseudographs does not affect the two problems in the same way.

In pseudographs, a multi-edge can never be in a matching cut, while self-loops have no impact. In contrast, multi-edges do not impact {\sc Stable Cut}. Whereas, motivated by a connection to surjective homomorphisms (as we discuss later), a {\it reflexive} vertex (i.e. with a self-loop) may not be in the stable cut. Hence, we consider the following two generalizations: {\sc Pseudograph Matching Cut} takes as input a pseudograph~$G$ (with no self-loops) and asks if $G$ has a matching cut, and {\sc Pseudograph Stable Cut} takes as input a pseudograph $G$ (with no multi-edges) and asks if $G$ has a stable cut.

\begin{figure}[b]
\centering \includegraphics[width=0.7\linewidth,page=1]{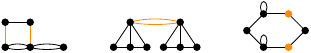}
\caption{Left: A pseudograph with a (unique) matching cut shown in orange. Centre: A pseudograph that has a (unique) $2$-cut shown in orange but no matching cut. Right: A pseudograph with a (unique) stable cut shown in orange.}
\end{figure}
For both problems, standard reductions can transform pseudograph instances into simple graph instances (for example, by replacing vertices or multi-edges with cliques \cite{GolovachPaulusmaSong} \cite{MatchingCut}). However, such transformations destroy the structure of monotone graph classes. This results in the difference in behaviour that we observe when comparing to simple graphs. 

The C123 framework of Johnson et al.~\cite{JMOPPSL25} 
initiated a systematic study of the complexity of graph problems on finitely defined monotone graph classes. Using known dichotomy results for specific problems~\cite{AK92,GP14,Ka12}, Johnson et al. introduced three conditions C1-C3, defined below, such that if a graph problem meets these conditions, then a widely applicable meta-classification holds. In this result, the set ${\cal S}$ consists of all graphs, every connected component of which is either a path or a subcubic tree with exactly one vertex of degree~$3$.

\begin{theorem}[Theorem 3 in \cite{JMOPPSL25}]\label{t-dicho}
For any \emph{finite} set of graphs ${\cal H}$, a \emph{C123}-problem $\Pi$  on ${\cal H}$-subgraph-free graphs is ``easy'' if ${\cal H}$ contains a graph from ${\cal S}$, or else it is ``hard''.
\end{theorem}

\noindent  
What constitutes ``easy'' and ``hard'' depends on the context. In \cite{JMOPPSL25}, it is usually {\sf P} versus \NP-complete, though sometimes it is almost linear versus quadratic. In \cite{EVMPS24}, the distinction also considers different levels in the Polynomial Hierarchy. In this paper, it will always be {\sf P} versus \NP-complete; see \cite{Leeuwen25} for a recent survey around the C123-framework.

Let us now define the conditions C1--C3  for simple graphs. A graph is {\it subcubic} if it has maximum degree at most~$3$. For $p\geq 1$, the {\it $p$-subdivision} of an edge $e=uv$ of a graph~$G$ replaces $e$ by a path of $p+1$ edges with endpoints $u$ and~$v$. The {\it $p$-subdivision} of a graph~$G$ is the graph obtained from~$G$ after $p$-subdividing each edge. For a graph class ${\cal G}$ and an integer~$p$, let ${\cal G}^p$ be the class consisting of the $p$-subdivisions of the graphs in ${\cal G}$. A graph problem $\Pi$ is ``hard'' {\it under edge subdivision of subcubic graphs} if for every $j \geq 1$ there is an~$\ell \geq j$ such that: if $\Pi$ is ``hard'' for the class ${\cal G}$ of subcubic graphs, then $\Pi$ is ``hard'' for ${\cal G}^{\ell}$. We can now say that a graph problem~$\Pi$ has property:\\[-9pt]
\begin{itemize}
\item {C1} if $\Pi$ is ``easy'' for every graph class of bounded tree-width;\\[-12pt]
\item {C2} if $\Pi$ is ``hard'' for subcubic graphs;\\[-12pt]
\item {C3} if $\Pi$ is  ``hard'' under edge subdivision of subcubic graphs;\\[-9pt]
\end{itemize}
For pseudographs, conditions C1 and C2 stay the same. In particular, within the context of this paper, the {\it degree} of a vertex~$u$ in a pseudograph $G$ is the number of neighbours of $u$ in $G$.  That is, multiple edges between $u$ and some other vertex $v$ are only counted once, and a possible self-loop $uu$ does not add to the degree of $u$ either. Hence, a pseudograph is {\it subcubic} if it is $K_{1,4}$-subgraph-free,  where $K_{1,4}$ is the $5$-vertex star.

We generalise condition C3 as follows. A multi-path is a sequence of vertices $v_1, \dots, v_k$ such that there is at least one (possibly several) edges between $v_i$ and $v_{i+1}$ for every $i \in \{1, \ldots, k-1\}$. Further, the vertices $v_1, \dots, v_k$ may or may not have self-loops. A {\it multi-path subdivision} of an edge $uv$ in a pseudograph~$G$ replaces $uv$ with a new vertex $w$ such that $u,w,v$ is a multi-path. For an integer $p\geq 1$, a {\it $p$-multi-path subdivision} of an edge $e=uv$ of a pseudograph~$G$ replaces $e$ by a multi-path of length $p+1$ with endpoints $u$ and~$v$. A {\it $p$-multi-path subdivision} of a pseudograph~$G$ is a graph obtained from~$G$ after $p$-multi-path subdividing each edge. We argue that the introduction of multi-paths is natural, as otherwise subdividing a multi-graph~$G$ would produce a simple graph. Note that for every $p \geq 1$, there are an infinite number of multi-paths of length~$p+1$. Hence, for every pseudograph~$G$, there are an infinite number of graphs that are the $p$-multi-path subdivision of $G$. Condition C3 is now as before, but with respect to our definition of a $p$-multi-path subdivision of a pseudograph.

As it is straightforward to adapt the proof of Theorem~\ref{t-dicho} to pseudographs, we obtain:

\begin{theorem}\label{t-dicho2}
For any finite set of simple graphs ${\cal H}$, a \emph{C123}-problem $\Pi$  on ${\cal H}$-subgraph-free pseudographs is ``easy'' if ${\cal H}$ contains a graph from ${\cal S}$, or else it is ``hard''.
\end{theorem}

We now turn to our testbed problems. We first consider {\sc Pseudograph Matching Cut}. As a multi‑edge cannot appear in a matching cut, we show that the complexity of this problem differs significantly from its restriction {\sc Matching Cut} to simple graphs. In simple graphs, almost all subcubic instances have a matching cut \cite{Moshi89}. This results in a polynomial-time algorithm for {\sc Matching Cut} on subcubic graphs. Hence, {\sc Matching Cut} does not satisfy condition C2 of the C123 framework. However, in contrast, {\sc Pseudograph Matching Cut}  is \NP‑complete even on subcubic pseudographs~\cite{MatchingCut}. We prove that {\sc Pseudograph Matching Cut} satisfies C1 and C3. Hence, by Theorem~\ref{t-dicho2}, we obtain the following classification:

\begin{theorem}\label{t-1}
For any finite set of simple graphs ${\cal H}$, {\sc Pseudograph Matching Cut} on ${\cal H}$-subgraph-free pseudographs is polynomial-time solvable if ${\cal H}$ contains a graph from ${\cal S}$; else it is \NP-complete.
\end{theorem}

\noindent
Theorem~\ref{t-1} stands in sharp contrast to the simple graph version, {\sc Matching Cut}, whose complexity remains poorly understood even for monotone classes defined by a single forbidden subgraph~$H$ (see~\cite{LozinMPPSSL24}). This phenomenon observed when moving to pseudographs parallels the shift from {\sc Colouring} to {\sc List Colouring}. While {\sc Colouring} is not a C123-problem~\cite{EJLMP}, the additional expressibility enabled by list constraints mean that {\sc List Colouring} is a C123-problem \cite{GP14,JMOPPSL25}. For {\sc Matching Cut}, adding a multi‑edge in a pseudograph is equivalent to imposing the constraint that the corresponding edge cannot belong to the cut. Theorem~\ref{t-1} shows that the local expressive power given by multi-edges resembles the power that list constraints provide for the colouring problem.

We will show that the same classification as Theorem~\ref{t-1} also holds for the pseudograph variant of the more general problem {\sc $d$-Cut}. In this problem, $d$ is a fixed integer, and the question is whether a given simple graph $G$ contains a {\it $d$-cut}, that is, an edge cut such that every vertex has at most $d$ neighbours on the opposite side of the cut. Observe that {\sc Matching Cut} is exactly {\sc $d$-Cut} with $d=1$.

Gomes and Sau~\cite{GS21} proved that for every $d \geq 1$, every graph~$G$ on at least eight vertices with maximum degree at most $d+2$ has a $d$-cut, which can be found in polynomial time. Hence, {\sc $d$-Cut} does not satisfy C2, which we show is in contrast to {\sc Pseudograph $d$-Cut}. In fact, we can extend Theorem~\ref{t-1} from $d=1$ to any fixed $d\geq 2$:

\begin{theorem}\label{t-2}
For any integer $d\geq 2$ and any finite set of simple graphs ${\cal H}$, {\sc Pseudograph $d$-Cut} on ${\cal H}$-subgraph-free pseudographs is polynomial-time solvable if ${\cal H}$ contains a graph from ${\cal S}$; else it is \NP-complete.
\end{theorem}

\noindent
We note that in contrast to Theorems~\ref{t-1} and~\ref{t-2}, the complexities of {\sc Matching Cut} and {\sc $d$-Cut} do not coincide even on $H$-subgraph-free graphs. For example, {\sc Matching Cut} is \NP-complete on $K_{1,5}$-subgraph-free graphs, that is, graphs with maximum degree at most~$4$~\cite{Chvatal1984} whereas, for every $d \geq 2$, {\sc $d$-Cut} can be solved in polynomial time on such graphs.

We now consider {\sc Stable Cut}. As {\sc Stable Cut} can be solved in polynomial time on subcubic graphs \cite{BDLS00}, {\sc Stable Cut} is not a C123-problem. We show that {\sc Pseudograph Stable Cut} is also not a C123-problem. Despite this, when restricted to graph classes defined by a single forbidden subgraph, the behaviour of {\sc Pseudograph Stable Cut} is consistent with that of a C123‑problem. However, this similarity already disappears when considering graph classes defined by two forbidden subgraphs. That is, there exist classes defined by just two forbidden subgraphs on which the complexity of {\sc Pseudograph Stable Cut} differs from C123‑problems. This provides the first example of a problem whose behaviour mirrors that of a C123 problem where a single subgraph is forbidden but not when there is a larger finite set of forbidden subgraphs.

\begin{theorem}\label{t-3}
For any simple graph~$H$, {\sc Pseudograph Stable Cut} on $H$-subgraph-free pseudographs is polynomial-time solvable if $H\in {\cal S}$; else it is \NP-complete. However, there exist sets ${\cal H}$ of two simple graphs, neither of which is in ${\cal S}$, such that  {\sc Pseudograph Stable Cut} on ${\cal H}$-subgraph-free pseudographs is polynomial-time solvable. 
\end{theorem}

Motivated by the connection between {\sc Stable Cut} and surjective homomorphisms, we consider a selection of other surjective homomorphism problems. A \emph{homomorphism} from a graph~$G$ to a graph~$H$ is a function $h:V(G)\rightarrow V(H)$ so that, for all  $xy \in E(G)$, we have $h(x)h(y) \in E(H)$. A homomorphism is \emph{surjective} if $h$ is a surjective function. Note that this definition also makes sense for graphs with loops permitted. The problem  {\sc Stable Cut} is exactly the problem to determine if a simple graph has a surjective homomorphism to $P_3^{101}$, where $P_3^{101}$ is the path on three vertices in which both leaves are looped and the interior vertex is loopless. Let us term {\sc SurHom}$(H)$ to be the problem defined by a fixed target graph $H$, which asks if a simple input graph $G$ has a surjective homomorphism to $H$. Such problems are well-studied in the  literature~\cite{GolovachPaulusmaSong,MartinP15} (see also the survey~\cite{BodirskyKM12}). Among the most famous of these problems is $\RSurHom(C^*_4)$, where $C^*_4$ is the reflexive $4$-cycle. This problem is also known as {\sc Disconnected Cut} (and the complement of {\sc $2K_2$-Partition}) and was prominent in the literature until its complexity was settled as \NP-complete in \cite{MartinP15}\footnote{Note that Narayan Vikas has an independent proof which is unpublished.}. The complexity of $\RSurHom(C^*_k)$ is still open for $k \in \{5,6\}$, but for all $k \geq 7$ it is \NP-complete~\cite{Russian}. 

\begin{table}[t]
\begin{center}
\begin{tabular}{|c|c|c|}
\hline
${\cal H}$ & {\sc Pseudograph Stable Cut} & {\sc Pseudograph Matching/ $d$-Cut}  \\
\hline  
$S_{j,k,\ell}$ & {\sf P} & {\sf P} \\  
\hline
$H_i^{j,k,\ell,m}$ & \NP-complete & \NP-complete \\    
\hline
$C_i$ & \NP-complete & \NP-complete \\    
\hline
$\{rN_{1,1,1},H_1^{1,1,1,1}\}$ & {\sf P} & \NP-complete \\
\hline
$\{H_1, N_{1,1,\ell}\}$ & {\sf P} & \NP-complete \\
\hline
$\{K_3,H_1^{2,2,2,1}\}$ & {\sf P} & \NP-complete \\ 
\hline
$\{K_3,H_1^{2,2,2,2}\}$ & \NP-complete & \NP-complete \\ 
\hline
\end{tabular}
\end{center}
\caption{Our  results for {\sc Pseudograph Matching Cut} / {\sc $d$-Cut} $(d\geq 2)$ and {\sc Pseudograph Stable Cut} for ${\cal H}$-subgraph-free instances (see Section~\ref{s-pre} for definitions).}
\label{tab:1}
\end{table}

\begin{table}[t]
\begin{center}
\begin{tabular}{|c|c|c|c|}
\hline
$H$ & {\sc Pseudograph Stable Cut} & {\sc Disconnected Cut} & {\sc Pseudograph} {\sc SurHom}($C'_k$)  \\
\hline  
$K_{1,4}$ & \NP-complete & {\sf P} \cite{FMPS09} & {\sf P}\\  
\hline
$K_3$ & \NP-complete & \hspace*{-7mm} {\sf P} (trivial) & \NP-complete \\    
\hline
\end{tabular}
\end{center}
\caption{Our results for {\sc Pseudograph Stable Cut} and {\sc Pseudograph} $\RSurHom(C'_k)$ $(k\geq~6)$ with a comparison to {\sc Disconnected Cut} for $H$-subgraph graphs (see Section~\ref{s-pre} for definitions).}\label{tab:2}
\end{table}

In contrast to {\sc Pseudograph Stable Cut}, the addition of self-loops to the input graph~$G$ does not impact the problem of $\RSurHom(C^*_k)$, as $C^*_k$ contains only reflexive vertices. However, this changes already if we delete exactly one self-loop. That is, we let $C_k'$ be obtained from the reflexive cycle $C_k^*$ by removing exactly one self-loop, and we consider the {\sc Pseudograph} $\RSurHom(C_k')$ problem. This problem has as input a pseudograph $G$, and asks if $G$ admits a surjective homomorphism to $C_k'$. Although not previously studied in the literature, {\sc Pseudograph} $\RSurHom(C_k')$ is polynomial-time solvable when $k\leq 4$ (this follows from the polynomial algorithm for the retraction problem, as shown in~\cite{Pseudoforests}). Interestingly, the behaviour of {\sc Pseudograph} $\RSurHom(C_k')$ appears to lie between $\RSurHom(C^*_k)$ and {\sc Pseudograph Stable Cut}. As shown by the following theorem, for triangle-free pseudographs, the complexity of {\sc Pseudograph} $\RSurHom(C_k')$ aligns with {\sc Pseudograph Stable Cut} whereas for $K_{1,4}$-subgraph-free it aligns with $\RSurHom(C^*_k)$.

\begin{theorem}\label{t-4}
For every $k\geq 6$, {\sc Pseudograph SurHom}($C'_k$) is \NP-complete even for triangle-free pseudographs. However, it is polynomial-time solvable for every class of pseudographs of bounded degree.
\end{theorem}

\noindent
To prove Theorem~\ref{t-4} we reduce from the problem {\sc Conflict-Free Cut}, which is a generalisation of {\sc Matching Cut} that was recently introduced by Rauch, Rautenbach and Souza~\cite{RAUCH2025106503} and that was also studied in~\cite{KK26}.

Tables~\ref{tab:1} and \ref{tab:2} summarise our results. We show Theorems~\ref{t-1}--\ref{t-4} in Section~\ref{s-1}--\ref{s-new}. In Section~\ref{s-con} we discuss some open problems.

\section{Preliminaries}\label{s-pre}

Let $G$ be a {\it pseudograph}, that is, $G$ may have multiple edges between two vertices $u$ and $v$ and may also have self-loops $uu$. We say that $u\in V(G)$ is {\it reflexive} if $uu\in E(G)$, that is, $u$ has a (self-)loop; otherwise $u$ is {\it irreflexive}. If $G$ only has reflexive vertices, then $G$ is {\it reflexive}. If $G$ only has irreflexive vertices, then $G$ is {\it irreflexive}. We may say that $G$ is {\it partially reflexive} to express that $G$ may have self-loops. If $G$ is irreflexive and may only have multiple edges, then $G$ is a \textit{multi-graph}. Moreover, $G$ is {\it simple} if it is irreflexive and has at most one edge between any two of its vertices. We say that $G$ is \emph{subcubic} if the maximum degree of the irreflexive graph underlying $G$ is at most~$3$, or equivalently, if $G$ is $K_{1,4}$-subgraph free. This is significant, as by taking the reduction of Patrignani and Pizzonia \cite{MatchingCut} we can show that {\sc Pseudograph Matching Cut} is \NP-hard on subcubic graphs. In \cite{MatchingCut}, they prove \NP-hardness for graphs with maximum degree~$4$ by transforming a subcubic multi-graph into a simple graph with maximum degree~$4$.

For $r\geq 3$, let $C_r$ denote the simple cycle on $r$ vertices. The graph $K_3$ denotes the simple triangle. A graph is {\it triangle-free} if it is $K_3$-subgraph-free. Let $S_{j,k,\ell}$ indicate the simple {\it subdivided claw}, which is obtained by identifying one end of paths of length $j,k,\ell$, respectively. Let $kH$ indicate the disjoint union of $k$ copies of~$H$.
Let $N_{i,j,k}$ be the simple graph formed by adding pendant paths of length $i$, $j$ and $k$ to the vertices of a triangle. In particular, $N_{1,1,1}$ is the {\it net}.

 Let ${H_{i}}$ be the graph obtained by adding a path of length $i$ between the internal vertices of two vertex-disjoint copies, $x_1-y_1-x_2$ and $x_3-y_2-x_4$, of $P_3$. We define ${H^{j_1, j_2, j_3, j_4}_{i}}$ to be the graph obtained from $H_i$ by subdividing the edge incident to $x_k$ $j_k-1$ times for $1 \leq k \leq 4$. In particular, the case $j_1=j_2=j_3=j_4=1$ corresponds to $H_i$ and we omit the superscript. If $j_1=j_2=j_3=j_4=k$, then we may write $H^k_i$. For example, $H^{2,2,2,2}_1=H^2_1$. We draw $H_1$ and $H_3$ in Figure~\ref{fig:Hs}.

In a graph $G$, let $N(u)$ be the set of vertices adjacent to $u$ in $V(G)$ and $N[u]=N(u) \cup \{u \}$. Call a pair of vertices $u$ and $v$ \emph{twins} if $N(u)=N(v)$. If $N[u]=N[v]$, then call $u$ and $v$ \emph{true twins}.

In this paper, $\mathcal{H}$ always refers to a finite set of graphs.

\begin{figure}
    \centering
\includegraphics{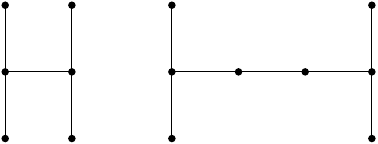}
\caption{The graphs $H_1$ and $H_3$.}
\label{fig:Hs}
\end{figure}

\section{Pseudograph Matching Cut}\label{s-1}

Recall that {\sc Pseudograph Matching Cut} asks if a pseudograph~$G$ has an edge cut that is a matching, hence a multi-edge can never be in the edge cut. This problem is equivalent to asking if~$G$ can be coloured with two colours such that each vertex is incident to at most one edge with differently coloured endpoints and each colour is used at least once.

\begin{proposition}\label{mc-bounded-tw}
{\sc Pseudograph Matching Cut} is polynomial time solvable on every class of pseudographs with bounded treewidth, so satisfies C1.
\end{proposition}

\begin{proof}
It was noted in \cite{B09} that {\sc Matching Cut} can be solved in polynomial time on graphs of bounded treewidth. To show the same for {\sc Pseudograph Matching Cut} first notice in an instance of {\sc Matching Cut}, we can first remove any self-loops. Further, if the multiplicity of an edge is at least two, both endpoints must be on one side of the cut meaning  additional edges may be removed leaving only multiplicities $1$ and $2$. As observed in \cite{MatchingCut}, by replacing each multiedge by a triangle we obtain a simple graph which has a matching cut if and only if the original pseudograph had a matching cut. This simple graph has at most a polynomial number of additional vertices, and maintains bounded treewidth, so {\sc Pseudograph Matching Cut} is also polynomial time solvable on pseudographs of bounded treewidth. \qed
\end{proof}

\begin{proposition}[\cite{MatchingCut}]\label{mc-subcubic}
{\sc Pseudograph Matching Cut} is \NP-complete on subcubic pseudographs, so satisfies C2.
\end{proposition}

The {\sc NAE 3-SAT} problem asks whether a 3-CNF formula $\phi$ has an assignment in which each clause contains at least one false literal and at least one true literal. Given an instance of {\sc NAE 3-SAT}, which is satisfied both when each variable is assigned to be true and when each variable is assigned to be false, we call the problem of deciding if there is a third satisfying assignment {\sc NAE 3-SAT 0-1}. Equivalently, {\sc NAE 3-SAT 0-1} asks whether such a formula has a satisfying assignment where some pair of variables take different values.

\begin{lemma}\label{l-hardhard}
{\sc NAE 3-SAT 0-1} is \NP-complete, even for instances in which every clause contains exactly one negative literal.
\end{lemma}

\begin{proof}
We reduce from {\sc Matching Cut}. Given an instance $G$ of {\sc Matching Cut}, we create an instance $\phi$ of {\sc NAE 3-SAT 0-1} with one variable $x_i$ for each vertex $v_i$ of $G$. For each vertex $v_i$ and each pair of neighbours $v_j,v_k$ of $v_i$, we add a clause $(\bar{x}_i,x_j,x_k)$. The formula $\phi$ is satisfied when every variable is true since every clause contains both positive and negative literals. The same applies when every variable is false. Additionally, every clause of $\phi$ contains exactly one negative literal.

We first assume that $G$ has a matching cut $(A,B)$. A further (third) satisfying assignment, $s$, for $\phi$ is obtained by setting variables in $A$ to true and $B$ to false. By definition, both $A$ and $B$ are non-empty. This implies that $s$ contains at least one true and at least one false variable. Assume that $s$ is not a satisfying assignment. Then there exists a clause in which all three literals take the same value. In other words, there is a clause $(\bar{x}_i,x_j,x_k)$ such that $x_i$ belongs to one side of the cut and both $x_j$ and $x_k$ belong to the other. This is a contradiction since $(A,B)$ is a matching cut and $v_i$ is adjacent to both $v_j$ and $v_k$.

Now assume that $\phi$ has a satisfying assignment $s$ with at least one true variable and at least one false variable. We obtain a matching cut $(A,B)$ of $G$ by letting $A$ be the set of vertices whose corresponding variables are set to true and $B$ be the set of vertices whose corresponding variables are false. Both $A$ and $B$ are non-empty since $s$ contains at least one true and at least one false variable. No vertex in $A$ has more than one neighbour in $B$ and vice versa since this leads to an unsatisfied clause in $\phi$. Specifically, if a vertex $v_i$ of $A$ has two neighbours $v_j$ and $v_k$ in $B$ then the clause $(\bar{x_i}, x_j, x_k)$ has three false variables. If $v_i \in B$ has two neighbours $v_j$ and $v_k \in A$ then this same clause has three true literals. \qed
\end{proof}

\begin{proposition}\label{lem:mmc-k-sub}
For all $k\geq 1$, {\sc Pseudograph Matching Cut} is \NP-complete for $k$-multi-path subdivisions of subcubic pseudographs, so satisfies (C2 and) C3.
\end{proposition}

\begin{proof}
We reduce from {\sc NAE 3-SAT 0-1} where every clause contains exactly one negative literal. Given an instance $\phi$ of this problem, we construct a subcubic instance $G$ of {\sc Pseudograph Matching Cut} as follows. For each variable $x_i$, we take a path $P_i$ on $\ell$ vertices, where $\ell$ is the number of occurrences of $x_i$ in $\phi$. Make each edge on the path a double edge. For each clause $C_j= (\bar{x}_{j_1}, x_{j_2}, x_{j_3})$ and $1 \leq i \leq 3$, we choose some vertex $v_{j_i}$ on the path $P_{j_i}$. For each clause containing $x_{j_i}$, we choose a different vertex in $P_{j_i}$. We then add a vertex $v_{j_1}'$ adjacent to $v_{j_1}$, by a double edge. Finally, we add single edges between $v_{j_1}'$ and each of $v_{j_2}$ and $v_{j_3}$. See Figure~\ref{fig:1} for an example.

If $G$ has a matching cut then every vertex on each variable path receives the same colour and some pair of variable paths receive different colours. We obtain a satisfying assignment in which some pair of variables are assigned different values by setting vertices in one colour to true and the other colour to false. Assume that this does not lead to a satisfying assignment, then some vertex has two neighbours in the other colour, a contradiction.

Next, assume that $\phi$ has a satisfying assignment where two variables receive different values. We obtain a matching cut of $G$ by colouring paths corresponding to true variables blue and paths corresponding to false variables red. By extending the colouring along the double edges, this results in a matching cut between the red and the blue, else some clause contains three literals with the same value.

We construct an instance $G'$, which is a $k$-multi-path subdivision of a subcubic graph similarly, except that we subdivide $k$ times each edge in the underlying simple graph of $G$, before adding the specified multi-edges. We make every new edge a double edge except for the first edges on the two paths from $v_{j_1}'$ to $v_{j_2}$ and $v_{j_3}$ for each clause $C_j$. \qed
\end{proof}

\begin{figure}[h]
\centering
\resizebox{6cm}{!}{
\begin{tikzpicture}[every state/.style={draw,circle},node distance=3em, inner sep=0.1cm]
\node[state] (v1) at (0, -4) {};
\path[thick] (v1) edge [double] (-0.75, -4);
\node[state] (v2) at (1.5,-4) {$v_2$};
\node[state] (v3) at (3,-4) {};
\path[thick] (v3) edge [double] (3.75, -4);
\draw [decorate,decoration={brace,amplitude=5pt,mirror,raise=4ex}]
  (-0.5, -4) -- (3.5,-4) node[midway,yshift=-3em]{$P_2$};

\node[state] (v4) at (3,0) {};
\path[thick] (v4) edge [double] (2.25, 0);
\node[state] (v5) at (4.5,0) {$v_1$};
\node[state] (v6) at (6,0) {};
\path[thick] (v6) edge [double] (6.75, 0);
\draw [decorate,decoration={brace,amplitude=5pt,raise=4ex}]
  (2.5,0) -- (6.5,0) node[midway,yshift=3em]{$P_1$};

\node[state] (v7) at (6,-4) {};
\path[thick] (v7) edge [double] (6-0.75, -4);
\node[state] (v8) at (7.5,-4) {$v_3$};
\node[state] (v9) at (9,-4) {};
\path[thick] (v9) edge [double] (9.75, -4);
\draw [decorate,decoration={brace,amplitude=5pt,mirror,raise=4ex}]
  (5.5, -4) -- (9.5,-4) node[midway,yshift=-3em]{$P_3$};

\node[state] (v11) at (4.5,-2) {$v'_1$};
\path[thick] (v1) edge (v2)
                (v1) edge [double] (v2)
                (v2) edge [double] (v3)
                (v4) edge [double] (v5)
                (v5) edge [double]  (v6)
                (v7) edge [double] (v8)
                (v8) edge [double] (v9)
                (v11) edge [double] (v5)
                (v11) edge  (v2)
                (v11) edge  (v8);
                 \end{tikzpicture}
                 }
\resizebox{6cm}{!}{
                 \begin{tikzpicture}[every state/.style={draw,circle},node distance=3em, inner sep=0.1cm]

\node[state] (v1) at (0, -4) {};
\node[state] (v2) at (1.5,-4) {$v_2$};
\node[state] (v3) at (3,-4) {};

\draw[thick, double] (v1) -- node[midway,circle,draw,inner sep=2pt,fill=white] (s12) {} (v2);
\draw[thick,double] (v1) -- (-0.75,-4);
\draw[thick,double] (v2) -- node[midway,circle,draw,inner sep=2pt,fill=white] (s23d) {} (v3);
\draw[thick,double] (v3) -- (3.75,-4);

\draw [decorate,decoration={brace,amplitude=5pt,mirror,raise=4ex}]
  (-0.5, -4) -- (3.5,-4) node[midway,yshift=-3em]{$P_2$};

\node[state] (v4) at (3,0) {};
\node[state] (v5) at (4.5,0) {$v_1$};
\node[state] (v6) at (6,0) {};

\draw[thick,double] (v4) -- (2.25,0);
\draw[thick,double] (v4) -- node[midway,circle,draw,inner sep=2pt,fill=white] (s45) {} (v5);
\draw[thick,double] (v5) -- node[midway,circle,draw,inner sep=2pt,fill=white] (s56) {} (v6);
\draw[thick,double] (v6) -- (6.75,0);

\draw [decorate,decoration={brace,amplitude=5pt,raise=4ex}]
  (2.5,0) -- (6.5,0) node[midway,yshift=3em]{$P_1$};

\node[state] (v7) at (6,-4) {};
\node[state] (v8) at (7.5,-4) {$v_3$};
\node[state] (v9) at (9,-4) {};

\draw[thick,double] (v7) -- (6-0.75,-4);
\draw[thick,double] (v7) -- node[midway,circle,draw,inner sep=2pt,fill=white] (s78) {} (v8);
\draw[thick,double] (v8) -- node[midway,circle,draw,inner sep=2pt,fill=white] (s89) {} (v9);
\draw[thick,double] (v9) -- (9.75,-4);

\draw [decorate,decoration={brace,amplitude=5pt,mirror,raise=4ex}]
  (5.5, -4) -- (9.5,-4) node[midway,yshift=-3em]{$P_3$};

\node[state] (v11) at (4.5,-2) {$v'_1$};

\draw[thick] (v11) -- node[midway,circle,draw,inner sep=2pt,fill=white] (s11v2) {} (v2);
\draw[thick,double] (v2) -- (s11v2);
\draw[thick] (v11) -- node[midway,circle,draw,inner sep=2pt,fill=white] (s11v8) {} (v8);
\draw[thick,double] (v8) -- (s11v8);
\draw[thick,double] (v11) -- node[midway,circle,draw,inner sep=2pt,fill=white] (s11v5) {} (v5);

\end{tikzpicture}
                 }
\caption{The clause $C_1=(\bar{x_1}, x_2, x_3)$ in the case $k=1$. Left: $G$. Right: $G'$. Note the horizontal paths belong to the variable gadgets. There are three edges connecting the clause gadget to its corresponding variable gadgets.}\label{fig:1}
\end{figure}

\noindent
Theorem~\ref{t-1} now follows from Propositions~\ref{mc-bounded-tw} and \ref{lem:mmc-k-sub}.

\section{Pseudograph $d$-Cut}\label{s-2}

In this section we extend our results for {\sc Pseudograph Matching Cut} (that is, {\sc Pseudograph $1$-Cut}) to {\sc Pseudograph $d$-Cut} for every fixed $d\geq 2$. We omit the proof of the following proposition that shows C1.

\begin{proposition}\label{prop:multi-d-cut-bounded-treewidth}
For $d\geq 2$, {\sc Pseudograph $d$-Cut} is polynomial-time solvable on every class of pseudographs with bounded treewidth, so satisfies C1.
\end{proposition}

\noindent
We now show C2 and C3.

\begin{proposition}\label{prop:m-d-c-subcubic}
For any $d \geq 1$, {\sc Pseudograph $d$-Cut} is \NP-complete for $k$-multi-path subdivisions of subcubic pseudographs, for any $k$, so satisfies C2 and C3.
\end{proposition}

\begin{proof}
Consider the graph $G'$ constructed in the proof of Proposition~\ref{lem:mmc-k-sub}. We construct a graph $G''$ such that $V(G'') = V(G')$ and edges of $G'$ with multiplicity~$1$ and~$2$ become edges with multiplicity $d$ and $d+1$, respectively in $G''$. Note $G''$ has a {\sc Pseudograph $d$-Cut} if and only if $G'$ has a {\sc Pseudograph Matching Cut}. \qed
\end{proof}

\noindent
Theorem~\ref{t-2} now follows from Propositions~\ref{prop:multi-d-cut-bounded-treewidth} and \ref{prop:m-d-c-subcubic}.

\section{Pseudograph Stable Cut}\label{s-3}

For a (multi-)graph $G$, the line graph of $G$, $L(G)$, has a vertex for each edge in $E(G)$. For each pair of edges of $G$ with at least one common endpoint, there is an edge between their corresponding vertices in $L(G)$. In particular, $L(G)$ is simple and irreflexive even if $G$ is a multi-graph. The following relationship between matching cut and stable cut holds for simple graphs.

\begin{theorem}[\cite{BDLS00}]
Let $G$ be a simple graph. If $L(G)$ contains a stable cut, then $G$ contains a matching cut; and if $G$ contains a matching cut and has minimum degree at least~$2$, then $L(G)$ contains a stable cut.
\end{theorem}

\noindent
The same result holds when $G$ is a multi-graph, except that each vertex must be incident to at least two edges as opposed to having an underlying degree of at least $2$.  

\begin{lemma}\label{lem:geqLG}
Let $G$ be a multi-graph. If $L(G)$ contains a stable cut, then $G$ contains a matching cut. If $G$ contains a matching cut and each vertex is incident to at least two edges (including multi-edges)\footnote{This is different from how we count double edges in degree.}, then $L(G)$ contains a stable cut.
\end{lemma}

Let $G_{s}$ be the underlying simple graph of some pseudograph $G$. Let $L^*(G)$ be obtained from $L(G_{s})$ by adding a self-loop to every vertex in $L(G_{s})$ corresponding to some edge with multiplicity at least $2$ in $G$. For a partially reflexive graph $G$, if $G$ contains a pair $u$, $v$ of true twins with degree $\geq 3$, then neither $u$ or $v$ can be contained in a minimal stable cut. Combined with Lemma~\ref{lem:geqLG}, this gives the following relationship between {\sc Pseudograph Matching Cut} and {\sc Pseudograph Stable Cut}.

\begin{corollary}\label{cor:main-reduction-matching-stable}
Let $G$ be a multi-graph. If $L^*(G)$ contains a (pseudograph) stable cut, then $G$ contains a (pseudograph) matching cut. If $G$ contains a (pseudograph) matching cut and each vertex is incident to at least two edges (counting double edges twice), then $L^*(G)$ contains a (pseudograph) stable cut.
\end{corollary}

\noindent
We now leverage this result to show the following (proof omitted). 

\begin{theorem}\label{5.1}
 {\sc Pseudograph Stable Cut} is:
\begin{itemize}
    \item  \NP-complete for $(C_4, C_5, \ldots, C_{\ell}, H_1, H_2 \dots H_k, K_{1,4})$-subgraph-free pseudographs for any $k \geq 1, \ell \geq 4$.
    \item \NP-complete for $(K_3, C_4, \ldots, C_{\ell}, H_2 \dots H_k, K_{1,4})$-subgraph-free pseudographs for any $k \geq 2, l \geq 4$.
    \end{itemize}
    \label{thm:siani-main}
\end{theorem}

\noindent
Let $H^2_1$ be the graph on $10$ vertices built from $H_1$ by subdividing all four of its pendant edges once. Call $H'$ a \emph{pendant subdivision} of $H$ if the only edges that are subdivided (perhaps many times) contain a vertex of degree $1$. We omit the proof of our next result.

\begin{corollary}\label{cor:main-stable-bis}
For any finite set of simple graphs ${\cal H}$ which either contains no pendant subdivision of a net 
or contains no pendant subdivision of $H_1$. 
{\sc Pseudograph Stable Cut} on ${\cal H}$-subgraph-free pseudographs is ``easy'' if ${\cal H}$ contains a graph from ${\cal S}$, or else it is ``hard''. Additionally, {\sc Pseudograph Stable Cut} is \NP-complete on $(K_3,H_1^2)$-subgraph-free pseudographs.
\end{corollary}

\noindent
By the previous corollary, the sets $\cal{H}$ in which the classification for {\sc Pseudograph Stable Cut} differs from a C123-problem must contain at least two graphs: one a pendant subdivision of a net and one pendant subdivision of $H_1$. We already saw that the $(K_3,H_1^2)$-subgraph-free case is \NP-complete. Specifically, the polynomial polynomial results of Table~\ref{tab:1}, whose proofs we omit, combined with Corollary~\ref{cor:main-stable-bis}, give Theorem \ref{t-3}.

\section{Surjective Homomorphism to a Cycle with a Single Loopless Vertex}\label{s-new}

To prove our next result, we rely on the problem {\sc Conflict-Free Cut}, a generalisation of {\sc Matching Cut} studied in \cite{RAUCH2025106503}. Here, an instance consists of a graph $G$ together with a graph $\hat{G}$, called a {\it conflict} graph, such that $V(\hat{G})=E(G)$. Note that two adjacent edges of $\hat{G}$ need not share any endpoints in $G$. A {\it conflict-free cut} of $G$ with respect to $\hat{G}$ is a set $F$ of edges of $G$ such that $F$ is an edge-cut of $G$ that forms an independent set of vertices in $\hat{G}$. We omit the proof of our next result.

\begin{theorem}
\label{thm:C'r-triangle-free}
For every $r\geq 6$, {\sc Pseudograph SurHom}($C'_r$) is \NP-complete for triangle-free pseudographs.
\end{theorem}

\noindent
We now show the following lemma.

\begin{lemma}\label{lem-SurHom-C'-diam}
    {\sc Pseudograph SurHom}($C'_r$) is trivial on pseudographs with diameter at least $2r-1$.
\end{lemma}

\begin{proof}
If the input is fully reflexive then it is a no-instance. We claim all other instances are yes-instances. Let $d$ be a loopless vertex of the input $G$. Let us partition the vertices of $G$ into sets  $V_i$ with distance $i$ from $d$. Since the diameter of $G$ is at least $2r-1$, each of $V_1,\ldots,V_{r-1}$ is non-empty. Consider the map which sends $d$ to $0$, $V_i$ to $i$ (for $i\leq r-1$) and $V_i$ to $r-1$ (for $i\geq r$). This is plainly a surjective homomorphism from $G$ to $P'_r$ (a path on $r$ vertices in which only one end vertex is loopless) and consequently to $C'_r$. \qed
\end{proof}

\begin{corollary}
{\sc Pseudograph SurHom}($C'_r$) is in {\sf P} on all classes of pseudographs with bounded-degree.
    \label{cor:C'r-in-P}
\end{corollary}

\begin{proof}
Let $\mathcal{G}$ be a class of pseudographs with bounded-degree and $G \in \mathcal{G}$. As there are a constant number of graphs with bounded degree and diameter, we can brute force a solution on these graphs. Else, we apply Lemma~\ref{lem-SurHom-C'-diam} as $G$ has diameter at least $2r-1$. \qed
\end{proof}

\noindent
Theorem~\ref{t-4} follows directly from Theorem~\ref{thm:C'r-triangle-free} and Corollary~\ref{cor:C'r-in-P}.

\section{Final Remarks}\label{s-con}

We have shown to what extent the C123-framework can be leveraged on enriched graph problems involving multi-edges and loops, using several testbed problems. In particular, we have shown that {\sc Pseudograph Matching Cut} and {\sc Pseudograph $d$-Cut} fit this framework. However, our results show a markedly different behaviour from {\sc Pseudograph Stable Cut}. Owing to Corollary~\ref{cor:main-stable-bis}, we can see that {\sc Pseudograph Stable Cut} behaves like a C123-problem on ${\cal H}$-subgraph-free graphs unless ${\cal H}$ contains both some pendant subdivision of the net and some pendant subdivision of $H_1$. Thus, the differing behaviour comes from precisely such graphs. In our paper we have investigated this phenomenon more closely. In particular, the third point of Corollary~\ref{cor:main-stable-bis} and Theorem~\ref{thm:tala-H2221} truly isolates the boundary of tractability, since $H^{2,2,2,1}_1$ is one vertex away from $H^2_1$.  One outstanding question concerns whether {\sc Pseudograph Stable Cut} is in P on $(K_3,H^{k,k,k,1}_1)$-subgraph-free graphs, for all $k>1$. We leave the complexity of {\sc Pseudograph} $\RSurHom(C'_5)$ open both for general pseudographs and for triangle-free pseudographs, just like the complexity of $\RSurHom(C'_k)$ on simple graphs is still open for $k\geq 5$. Finally, it would be interesting which other natural graph problems defined on pseudographs fit in this framework.

\subsection*{AI Use}
Generative AI was not used in this work.

\bibliography{local}

\appendix

\section{The Proof of Proposition~\ref{prop:multi-d-cut-bounded-treewidth}}\label{a-first-1}

Here is the missing proof of Proposition~\ref{prop:multi-d-cut-bounded-treewidth}.

\medskip
\noindent
\textbf{Proposition~\ref{prop:multi-d-cut-bounded-treewidth}.}
For $d\geq 2$, {\sc Pseudograph $d$-Cut} is polynomial time solvable on every class of pseudographs with bounded treewidth, so satisfies C1.

\begin{proof}
When $d$ is a fixed integer, {\sc $d$-Cut} can be solved in polynomial time on classes of pseudographs with bounded treewidth \cite{GS21}. To show the same for {\sc Pseudograph $d$-Cut}, we show that for every pseudograph $G$ we can construct a simple graph $G'$ such that $G'$ contains a $d$-cut if and only if $G$ contains a pseudograph $d$-cut, and $\mathit{tw}(G') \leq \mathit{tw}(G)(2d+1)$. Recall that we can remove all self-loops from $G$. 

First observe that if an edge of $G$ has multiplicity at least $d+1$, then its endpoints must lie on the same side of any pseudograph $d$-cut. Hence, we may assume that every edge of $G$ has multiplicity at most $d+1$.

Let us now construct the graph $G'$. For every vertex $v \in V(G)$, $G'$ contains a a clique $C^v$ of size $2d+1$. Let $C^v = \{v_1, \ldots, v_{2d+1}\}$. We now add the following edges between these cliques. For every edge $uv$ of multiplicity $r$ in $G$, add the edges $v_1u_i$ for every $1\leq i \leq r$ and $u_1v_i$ for every $2 \leq i \leq r$.

Let us now claim that $\mathit{tw}(G') \leq (2d+1)\mathit{tw}(G)$. Let $\mathcal{T} = (T,X)$ be a tree decomposition of $G$ with width $\mathit{tw}(G)$. We now describe a tree decomposition $\mathcal{T}' = (T,X')$ of $G'$. Intuitively, we obtain $\mathcal{T}'$ from $\mathcal{T}$ by, in each bag, replacing every vertex $v$ by the set $C^v$. Formally, $X'(b)=\bigcup_{v\in X(b)} C^v$, for every bag $b \in V(T)$. By construction, $\mathcal{T}'$ has width $(2d+1)\mathit{tw}(G)$. It remains to verify that $\mathcal{T}'$ is a tree decomposition of $G'$. Let $w$ be a vertex of $G'$ and let $v$ be that vertex in $G$ such that $w \in C^v$. By construction, the bags containing $w$ are exactly the bags that contained $v$ in $\mathcal{T}$. Since these bags induce a connected subtree of $T$, the same holds for $w$. Hence, it remains to show that every edge of $G'$ is contained in some bag of $\mathcal{T}'$. Consider some edge $ww'$ of $G'$. By construction, there exist vertices $v,v'$ of $G$ such that $w\in C^v$ and $w'\in C^{v'}$. If $v=v'$, then every bag containing $w$ also contains $w'$. Hence, some bag contains both endpoints. If $v\neq v'$, then by construction of $G'$, $v$ and $v'$ are adjacent in $G$. Since $\mathcal{T}$ is a tree decomposition of $G$, there is a bag $b$ containing both $v$ and $v'$. It follows that $X'(b)$ contains the vertices $C^v \cup C^{v'}$, and therefore contains both $w$ and $w'$. Thus every edge of $G'$ is contained in some bag of $\mathcal{T}'$ and so $\mathcal{T}'$ is a tree decomposition of $G'$.

It remains to show that $G'$ has a $d$-cut if and only if $G$ has a pseudograph $d$-cut. Suppose first that $G$ has a pseudograph $d$-cut. That is, there is some red-blue colouring $c$ of $G$ such that every vertex is incident to at most $d$ bichromatic edges (where an edge is bichromatic iff its ends are coloured differently). Let $c'$ be that colouring of $G'$ such that, the vertices $C^v$ are coloured the same as $v$, for each $v \in V(G)$. We now claim that every vertex of $G'$ is incident to at most $d$ bichromatic edges under $c'$. By construction, if there is some vertex of $C^v$ that is incident to at least $d+1$ bichromatic edges under $c'$, then $v$ is also incident to at least $d+1$ bichromatic edges under $c$, a contradiction. Hence, every vertex of $G'$ is incident to at most $d$ bichromatic edges under $c'$ and so $c'$ describes a $d$-cut. 

Suppose now that $G'$ has a $d$-cut. That is, there is some red blue colouring $c$ of $G'$ such that every vertex is incident to at most $d$ bichromatic edges. As the vertices of $C^v$ induce a clique of size $2d+1$, the vertices of $C^v$ must be assigned the same colour under $c$. Let $c'$ be that colouring of $G$ such that $v$ is assigned the same colour as $C^v$. We now claim that every vertex of $G$ is incident to at most $d$ bichromatic edges. Let $v$ be some vertex of $G$. If there is some vertex $u$ such that there are exactly $r$ bichromatic edges between $v$ and $u$, then there are also exactly $r$ bichromatic edges between $v_1$ and vertices in $C^u$. Since there are at most $d$ bichromatic edges incident to $v_1$ in $G'$, there are at most $d$ bichromatic edges incident to $v$ in $G$. Hence, $c'$ describes a $d$-cut. \qed
\end{proof}

\section{The Proof of Theorem~\ref{thm:siani-main}.}\label{a-first-2}

Here is the missing proof of Theorem~\ref{thm:siani-main}.

\medskip
\noindent
\textbf{Theorem~\ref{thm:siani-main}.}
 {\sc Pseudograph Stable Cut} is:
\begin{itemize}
    \item  \NP-complete for $\{C_4, C_5, \ldots, C_{\ell}, H_1, H_2 \dots H_k, K_{1,4}\}$-subgraph-free pseudographs for any $k \geq 1, l \geq 4$.
    \item \NP-complete for $\{K_3, C_4, \ldots, C_{\ell}, H_2 \dots H_k, K_{1,4}\}$-subgraph-free pseudographs for any $k \geq 2, l \geq 3$.
    \end{itemize}

\begin{proof}
Like Proposition~\ref{lem:mmc-k-sub}, we reduce from {\sc NAE 3-SAT 0-1} where every clause contains exactly one negative literal. In fact, given an instance $\phi$ of this problem, we construct that graph $G$ from Proposition~\ref{lem:mmc-k-sub}. That is, $G$ is a yes-instance of {\sc Pseudograph Matching Cut} if and only if $\phi$ is a yes instance of {\sc NAE 3-SAT 0-1}. Hence, by Corollary~\ref{cor:main-reduction-matching-stable}, $L^*(G)$ has a partially reflexive stable cut if and only if $\phi$ is a yes instance of {\sc NAE 3-SAT 0-1}.

Let us subdivide the paths in $G$ in the variable gadgets (with multi-edges) so that every other vertex has degree~$2$. This ensures that $L^*(G)$ is subcubic. Note that every vertex in $G$ is incident to at least two edges (counting double edges twice). The resulting clause gadget (in $L^*(G)$) is drawn in Figure~\ref{fig:stable-cut}. The variable gadgets (in $L^*(G)$) are reflexive paths of triangles. Note that these reflexive triangles may be contracted to a single loop without affecting if the instance has a partially reflexive stable cut.

For the first part of the theorem, it remains to argue that the reduced instances $L^*(G)$ can be made $\{C_4, C_5 \ldots C_{\ell},  H_1,$ $H_2, \dots H_k, K_{1,4}\}$-subgraph-free. Let $q=\max\{k,l\}$. Let $G'$ be that graph obtained from $L^*(G)$ as follows. For every edge between two looped vertices in a variable gadget, subdivide this edge $q$ times. Suppose that $G'$ contains $H_i$ as a subgraph for some $1 \leq i \leq k$. The only vertices in $G'$ with degree at least $3$ are sitting in the variable gadgets and the three vertices of each clause triangle $T_{j}$. Since each edge of the variable gadgets was subdivided at least $k$ times, the distance between a degree $3$ vertex in a variable gadget and any other degree $3$ vertex is at least $k+1$. Any two vertices belonging to the same clause triangle have degree exactly three and share a neighbour. Therefore, we may assume that the degree $3$ vertices of $H_i$ belong to distinct triangles. This leads to a contradiction since any path between two such vertices goes through some edges in the variable gadget and hence has length at least $k+1$ (since the subdivision has taken place). Similarly, any cycle of length at least $4$ in $G'$ contains some vertex from a variable gadget. Therefore any such cycle has length at least $2(l+1)$. For $K_{1,4}$-subgraph-freeness, note that $G'$ has maximum degree $3$.

Finally, we adapt the construction of $G'$ to prove the second part of the theorem. Again, let $q=\max\{k,l\}$. We subdivide the edges of the variable gadgets in $G$ $q$ times and replace the clause triangles $T_j$ by cycles of length $2q+3$ which are reflexive except for exactly two adjacent vertices. We add edges between the two link points of the variable gadgets of the positive literals and the two loopless vertices in the $C_{2q+3}$, respectively. In addition, we add an edge between the link point of the negated literal and some reflexive vertex which is at distance at least $q+1$ from either irreflexive vertex on the cycle. By the same arguments as above, $G'$ has a stable cut if and only if $\phi$ has a third satisfying assignment.

Assume that $G'$ contains a cycle $C$ of length at most $\ell$. If $C$ involves any vertex from a variable gadget, then the length of $C$ is at least $2(q+1)$. Therefore, $C$ does not contain a vertex of any variable gadget. This leads to a contradiction since each clause gadget is a cycle of length $2q+3$ and there are no edges between any two distinct clause gadgets. Now assume that $G'$ contains a subgraph $H_i$ with $2 \leq i \leq k$. This $H_i$ subgraph does not contain any vertex from a variable gadget since $i \leq q$. Therefore both degree $3$ vertices of this subgraph belong to the same clause gadget. This leads to a contradiction since paths between two such vertices either have length $1$ or length at least $q+1$. \qed
\end{proof}

\begin{figure}[h]
\begin{center}
\includegraphics[width=0.8\linewidth,page=2]{IPE-figs-multi.pdf}
\end{center}
\caption{The clause $C_1=(\bar{x}_1, x_2, x_3)$ in the proof  of Theorem~\ref{thm:siani-main}. Vertices are labeled by their corresponding variables or clauses. Each vertex $x_i$ corresponds to an edge in the path of the variable vertex for $x_i$ in the {\sc Pseudograph Matching Cut} reduction. Each vertex $C_i= (\bar{x_i}, x_j, x_k)$ corresponds to an edge between $v_i'$ and the variable gadgets of $x_i$, $x_j$ and $x_k$ in this reduction.}\label{fig:stable-cut}
\end{figure}

\section{Proof of Corollary~\ref{cor:main-stable-bis}}\label{a-apa}

Here is the missing proof of Corollary~\ref{cor:main-stable-bis}.

\medskip
\noindent 
\textbf{Corollary~\ref{cor:main-stable-bis}.}
{\it For any finite set of simple graphs ${\cal H}$, if $\mathcal{H}$ does not contain both a pendant subdivision of a net and a pendant subdivision of $H_1$, then {\sc Pseudograph Stable Cut} on ${\cal H}$-subgraph-free pseudographs is ``easy'' if ${\cal H}$ contains a graph from ${\cal S}$, or else it is ``hard''. Additionally, {\sc Pseudograph Stable Cut} is \NP-complete on $\{K_3,H_1^2\}$-subgraph-free pseudographs.}

\begin{proof}
Let us give an argument that {\sc Pseudograph Stable Cut} in in P on classes of pseudographs with bounded treewidth. We claim it can expressed by an $\textrm{MSO}$ formula, whereupon we may appeal to Courcelle's Theorem \cite{Courcelle}. For example, if $I$ is a new variable holding on some set of vertices, then $\forall x,y \ (I(x) \wedge I(y)) \rightarrow \neg E(x,y)$ expresses that $I$ is an independent set, while
\[
\exists U \exists x, y \ U(x) \wedge \neg U(y) \wedge \forall x, y \ (U(x) \wedge \neg U(y) \wedge \neg I(x) \wedge \neg I(y)) \rightarrow \neg E(x,y)    
\]
(where $U$ is a new unary relation symbol) expresses that $E$ minus $I$ is disconnected.

If $\cal{H}$ contains any (disjoint union of) subdivisions of a path or a claw then the result clearly holds. Let us assume it does not.

If $\cal{H}$ contains a a pendant subdivision of $H_1$, but not a pendant subdivision of a net, then we may use the proof from the first part of Theorem~\ref{thm:siani-main}, taking care that our paths are long enough to omit any of the $H_1, H_2, H_3, \ldots, C_4, C_5, \ldots$ in ${\cal H}$.

If $\cal{H}$ contains only a pendant subdivision of the net, but no pendant subdivision of $H_1$, then we may use the proof from the second part of Theorem~\ref{thm:siani-main}. This construction is triangle-free so plainly omits any pendant subdivision of the net, though we need to ensure that our paths are long enough to omit any of the $H_2, H_3, \ldots, (K_3,) C_4, C_5, \ldots$ in ${\cal H}$.

For the $\{K_3,H_1^2\}$-subgraph-free case, let us return to the graph we had that was $\{C_4,C_5,\ldots,C_k,H_1,H_2,\ldots,H_k\}$-subgraph-free in the first case of Theorem~\ref{thm:siani-main}. It had many $K_3$s with one looped vertex. The idea is to subdivide only once, replacing these with $C_4$s, similarly to what we did in the proof of the second part of Theorem~\ref{thm:siani-main} (see Figure~\ref{fig:stable-C4} for guidance). The resulting graph omits both $K_3$ and $H_1^2$ as a subgraph. \qed
\end{proof}

\begin{figure}[h]
\begin{center}
$
   \xymatrix{
& \bigcirc  \ar@{-}[d]  & \\
& \bigcirc \ar@{-}[dl] \ar@{-}[dr] & \\
\bullet \ar@{-}[rr] \ar@{-}[d]   & & \bullet \ar@{-}[d] \\
\bigcirc & & \bigcirc \\
}
$
\hspace{2cm}
$
  \xymatrix{
\bigcirc  \ar@{-}[d]  & \\
\bigcirc \ar@{-}[d] \ar@{-}[rr] & & \bigcirc \ar@{-}[d] \\
\bullet \ar@{-}[rr] \ar@{-}[d]   & & \bullet \ar@{-}[d] \\
\bigcirc & & \bigcirc \\
}
$
\end{center}
\caption{The substitution of the $K_3$s (left) by $C_4$s (right) in the proof of Corollary~\ref{cor:main-stable-bis}. Note that open circles represent looped vertices.}
\label{fig:stable-C4}
\end{figure}

\section{The Proofs of Three Polynomial-Time Results from Table~\ref{tab:1}}\label{ap-2}

In this section we prove the three polynomial-time results of Table~\ref{tab:1} for {\sc Pseudograph Stable Cut} corresponding to rows 4--6.

We start with the following lemma. 

\begin{lemma}\label{lem:gen-obs}
There exists a polynomial time algorithm which, given an instance $G$ of {\sc Pseudograph Stable Cut}, either finds that $G$ is a yes-instance or returns an equivalent instance $G'$ with the following properties:
\begin{itemize}
    \item[1.] Every vertex of degree $2$ is reflexive and no vertex has degree ($0$ or) $1$.
    \item[2.] No degree $2$ vertex is contained in a triangle.
    \item[3.] There are no twin vertices.
    \item[4.] For any pair $u,v \in V$, if $N(u) \subseteq N(v)$, then $u$ must be reflexive and $v$ must not be.
\end{itemize}

Moreover, any simple graph contained as a subgraph in $G'$ is contained in $G$.
\end{lemma}

\begin{proof}
{\textbf {Property 1}}:
If a vertex $v$ of $G$ has degree ($0$ or) $1$ then either $N(v)$ is a stable cut of $G$ has a stable cut if and only if $G \setminus \{v\}$ does. Therefore, we either find that $G$ is a yes-instance or obtain an equivalent instance by deleting $v$.

Let $v$ be an irreflexive degree $2$ vertex. If $v$ is a cut vertex or $N(v)$ forms a stable cutset then $G$ is a yes-instance. Else, $G$ has a stable cutset if and only if $G \setminus {v}$ does and we may delete $v$ to obtain an equivalent instance. Repeating this process exhaustively, we either find that $G$ is a yes-instance or obtain an instance in which every vertex has degree at least $2$ and every degree $2$ vertex is reflexive.

\textbf{Property 2}:
Consider a degree $2$ vertex $v$ contained in a triangle. This vertex is reflexive by Property 1. Since the neighbourhood of $v$ is a clique, $N(v)$ is not a stable cut. Since $v$ is reflexive, it does not belong to any stable cut of $G$. Therefore, any stable cut of $G$ is a stable cut of $G \setminus v$. Since exactly one side of any stable cut in $G \setminus v$ contains a neighbour of $v$, $G$ has a stable cut if and only if $G \setminus v$ does and we may delete $v$. Repeating this process exhaustively, we obtain an instance in which every degree $2$ vertex has two independent neighbours.

\textbf{Property 3}:
Assume that $G$ contains a vertex $x$ with a true twin $y$. From Property $2$ we may assume $x$ has degree at least $3$. Assume that $G$ has a stable cut $C$ containing at least one of $x$ and $y$. It contains at most one of $x$ and $y$ since they are adjacent. Without loss of generality, $C$ contains $x$. Furthermore, since $x$ and $y$ have the same open neighboburhood, removing $x$ from $C$ does not connect the graph $G \setminus {C}$. Therefore, neither $x$ nor $y$ are contained in a minimal stable cut of $G$. Therefore, $G$ has a stable cut if and only if $G \setminus \{y \}$ has a stable cut which does not contain $x$. In other words, We obtain an equivalent instance by deleting $y$ and making $x$ reflexive.

We now assume that $x$ and $y$ are twins. If there exists a stable cut of $G$ containing exactly one, say $x$, of $x$ and $y$ then, since $x$ and $y$ have the same open neighbourhood, there exists a stable cut containing neither of them. Therefore, any minimal stable cut $C$ either contains both $x$ and $y$ or neither of the two. In particular, if either of $x$ and $y$ is reflexive then we may make the other reflexive. Now, $G$ has a stable cut if and only if $G \setminus \{x \}$ does. In other words, we obtain an equivalent instance by deleting $x$.

Repeating these steps exhaustively, we obtain a graph with no twin vertices.

\textbf{Property 4}:
Note that $u$ and $v$ are non-adjacent since $v \notin N(v)$. Assume that $u$ is irreflexive. If $N(u)$ forms a stable cut then $G$ is a yes-instance, Since every neighbour of $u$ is a neighbour of $v$, $u$ is not a cut vertex. Therefore, any stable cut of $G$ is a stable cut of $G \setminus u $. On the other hand, consider a stable cut $C$ of $G \setminus u$. If $v$ belongs to $C$ then, since $u$ is irreflexive, we may add $u$ to $C$ to obtain a stable cut of $G$. If $v$ does not belong to $C$ then we add $u$ to the same side as $v$ to obtain a stable cut of $G$. Therefore, if $u$ is irreflexive then $G$ has a stable cut if and only if $G \setminus \{u\}$ does and we may delete $u$ to obtain an equivalent instance.

Now assume that both $u$ and $v$ are reflexive. Since every neighbour of $u$ is a neighbour of $v$, $G$ has a stable cut if and only if $G \setminus u$ does and we may delete $u$ to obtain an equivalent instance.

Finally, since each of these properties is obtained through vertex deletion, any simple subgraph contained in $G'$ is contained in $G$. \qed
\end{proof}

The next lemma deals with three additional properties which apply specifically to $H_1$-subgraph-free instances.

\begin{lemma}\label{lem:H-obs}
There exists a polynomial time algorithm which, given any $H_1$-subgraph-free instance $G$ of {\sc Pseudograph Stable Cut}, either decides $G$ or returns an equivalent instance $G'$ with the following properties:
\begin{enumerate}
\item[h1.] Every irreflexive vertex is contained in a triangle with another irreflexive vertex.
\item[h2.] No two triangles share a vertex.
\item[h3.] Each vertex of any triangle has exactly one neighbour outside the triangle and this vertex has degree 2.
\end{enumerate}

Moreover, any simple graph contained as a subgraph in $G'$ is contained in $G$.
\end{lemma}

\begin{proof}
We may assume that $G$ is connected and satisfies Properties 1 to 4.

\textbf{Property h1}:
Let $S$ be the set of irreflexive vertices in $G$. Note that any stable cut $C$ of $G$ is contained in $S$. By Property 1, every vertex in $S$ has degree at least $3$ in $G$. Let $v$ be a vertex in $S$ with no neighbours in $S$. Either $v$ is a cut vertex and $G$ is a yes-instance, or $G$ has a stable cut if and only if $G \setminus v$ does. Therefore we may remove any vertex which is isolated in $G[S]$ and is not a cut vertex. If this process leaves $S$ empty then $G$ is a no-instance. Otherwise, by $H_1$-subgraph-freeness, any two adjacent vertices in $S$ are contained in a triangle.

\textbf{Property h2}:
First assume that $G$ contains a diamond subgraph with vertex set $D={x,y,z,w}$ where $x$ and $y$ have degree $2$ in the diamond. Assume that $x$ and $y$ are adjacent. Then, by $H_1$-subgraph-freeness, at most one vertex outside $D$ has a neighbour in $D$. If $G=K_4$ then $G$ is a no-instance. Otherwise, either the neighbour $v$ of $D$ forms a stable cut or $G$ has a stable cut if and only if $G \setminus D$ does.

Now assume that $x$ and $y$ are non-adjacent. By Property $2$, each of them has at least one further neighbour. By $H_1$-subgraph-freeness, each of $x$ and $y$ has exactly one further neighbour. Let $a$ be a neighbour of $x$ outside $D$ and $b$ a neighbour of $y$ outside $D$. Neither $z$ nor $w$ has any neighbours outside $D \cup \{a, b \}$ by $H_1$-subgraph-freeness. If $a=b$ then $x$ and $y$ are twins, contradicting Property $3$.

Otherwise, both $a$ and $b$ have degree $2$ by $H_1$-subgraph-freeness and must be reflexive by Property $1$. Additionally, neither $z$ nor $w$ has a neighbour outside $D$ by $H_1$-subgraph-freeness. In particular, $z$ and $w$ are true twins. By Property $4$ we obtain an equivalent instance $G'$ by deleting the vertex $z$.

Now assume that the set $T \subseteq V(G)$  forms a subgraph of $G$ consisting of two triangles sharing exactly one vertex $v$. By $H_1$-subgraph-freeness and Property 1, the only vertex of $T$ which may have a neighbour outside $T$ is $v$. In particular, since we may assume that $G$ contains no diamond subgraph, the remaining vertices of $T$ have degree $2$, contradicting Property $2$.

\textbf{Property h3}:
By Property $2$ every vertex of any triangle $T$ has a neighbour outside $T$. By Property h$2$ no vertex outside $T$ has more than one neighbour in $T$. By $H_1$-subgraph-freeness each neighbour of $T$ then has degree $2$.

Finally, since each of these properties is obtained through vertex deletion, any simple subgraph contained in $G'$ is contained in $G$. \qed
\end{proof}

\begin{theorem}
{\sc Pseudograph Stable Cut} is polynomial-time solvable for $\{H_1,$ $r N_{1,1,1}\}$-subgraph-free graphs for any $r \geq 1$.
\end{theorem}

\begin{proof}
 We show that any $rN_{1,1,1}$-subgraph-free graph has at most $12r-1$ irreflexive vertices. Assume otherwise. By Properties h$2$ and h$3$, this implies that $G$ contains at least $4r$ vertex disjoint triangles. Since, by Property h$3$, any such triangle shares a neighbour with at most three others, $G$ contains at least $r$ disjoint copies of $N_{1,1,1}$, contradicting $rN_{1,1,1}$-subgraph-freeness. Therefore, we may test each possible combination of irreflexive vertices to decide whether it forms a stable cut of $G$. \qed
\end{proof}

This yields the polynomial time proof of the 5th row of Table~\ref{tab:1}. Together with Corollary~\ref{cor:main-stable-bis}, the previous theorem yields a proof of Theorem~\ref{t-3}. We continue however, with a view to obtaining some further results.

The following theorem is strictly incomparable to the last.
\begin{theorem}
{\sc Pseudograph Stable Cut} is polynomial-time solvable for $\{H_1,$ $N_{1,1,\ell}\}$-subgraph-free graphs for any $l \geq 0$.
\end{theorem}

\begin{proof}
We may assume that $G$ contains a path $P$ of length $2\ell$ as a subgraph, else the problem is polynomial-time solvable since $P_k$-subgraph-free graphs have bounded treewidth. If every triangle in $G$ contains at least one vertex of $P$ then $G$ contains at most $2\ell$ triangles by Property h$2$ and hence at most $6\ell$ irreflexive vertices by Property h$1$. In this case we test every subset of these vertices to decide whether it forms a stable cut. Otherwise, assume $G$ contains at least $6\ell+1$ irreflexive vertices and consider a triangle $T={x,y,z}$ containing no vertex of $P$. Without loss of generality, assume that $x$ is an endpoint of a shortest path $Q$ from $T$ to $P$, intersecting $P$ at a vertex $v$. No vertex of $T$ has a neighbour on the interior of $P$ since any neighbour of $T$ has degree $2$ by Property h$3$. The vertex $v$ divides $P$ into two segments, one of which has length at least $\ell$. Therefore we may extend $Q$ to a path $Q'$ of length at least $\ell$ and obtain a subgraph $N_{1,1,\ell}$ by taking $T, Q'$ and two further neighbours of $y$ and $z$. This implies that $G$ contains at most $6\ell$ irreflexive vertices, a contradiction. \qed 
\end{proof}
This yields the polynomial time proof of the 6th row of Table~\ref{tab:1}.

The graph $H_1^{2,2,2,1}$ is obtained by subdividing $3$ out of the $4$ pendant edges of $H_1$ (see Figure~\ref{fig:subH1}).

\begin{figure}
    \centering
    \begin{tikzpicture}[thick]
        \node[draw, circle, fill, inner sep=2pt] (A) at (0,1) {};
        \node[draw, circle, fill, inner sep=2pt] (A1) at (0,0) {};
        \node[draw, circle, fill, inner sep=2pt] (B) at (0,-2) {};
        \node[draw, circle, fill, inner sep=2pt] (B1) at (0,-3) {};
        \node[draw, circle, fill, inner sep=2pt] (C) at (2,1) {};
        \node[draw, circle, fill, inner sep=2pt] (C1) at (2,0) {};
        \node[draw, circle, fill, inner sep=2pt] (D) at (2,-2) {};
        \node[draw, circle, fill, inner sep=2pt] (E) at (0,-1) {};
        \node[draw, circle, fill, inner sep=2pt] (F) at (2,-1) {};
        
        \draw (A) -- (A1) -- (B1) -- (B);
        \draw (C) -- (C1) -- (D);
        \draw (E) -- (F);
    \end{tikzpicture}
    \caption{ $H_1^{2,2,2,1}$}
    \label{fig:subH1}
\end{figure}
We now work towards a proof of the result on the penultimate line in Table~\ref{tab:1}. The proof relies on showing that, given some $\{H_1^{2,2,2,1}, K_3\}$-subgraph-free graph after applying Lemmas~\ref{lem:gen-obs} and \ref{lem:small-cut}, the set of irreflexive vertices, $S$, forms a stable set. This implies $G$ admits a (pseudograph) stable cut if, and only if $S$ is a cutset.

\begin{lemma}\label{lem:small-cut}
For any constant $k$, any instance of {\sc Pseudograph Stable Cut} can be reduced in polynomial time to an equivalent instance $G'$ such that for every $Z \subseteq V(G')$ with $\vert Z \vert \leq k$, if $G'[Z]$ is connected with maximum degree at least $3$, then $Z$ contains at least three vertices with some neighbour in $V(G')\setminus Z$.
\end{lemma}

\begin{proof}
For each $Z \subseteq V$ such that $3 \leq \vert Z \vert \leq k$ and $G[Z]$ is connected with maximum degree $\geq 3$. If $Z$ meets one of the following conditions, we either return that $G$ is a yes-instance or find some $G'$, with size less than $G$ such that $G'$ has a pseudograph stable cut if, and only if $G$ does. Further, every simple subgraph of $G'$ is contained in $G$.

\begin{itemize}
\item[\textbf{s1.}] If $Z$ contains some pseudograph stable cut for $G$ return yes.
\item[\textbf{s2.}] If there is some $v \in Z$ such that every path between $Z$ and $G \setminus Z$ contains $v$. We call this a gateway vertex. From \textbf{s1.} we may assume $Z$ does not contain some pseudograph stable cut meaning $v$ must be reflexive. Further, no vertex of $Z$ is contained in some minimal stable cut. Let $G' =G \setminus (Z \setminus v)$.
\item[\textbf{s3.}] If there is some pair $v,v' \in Z$ such that every path between $Z$ and $G \setminus Z$ contains either $v$ or $v'$. We call these a gateway vertices. We again may assume $Z$ does not contain some stable cut. Let $P$ be the shortest path from $v$ to $v'$ in $G[Z]$.

We first consider the case where there is some stable set $C \subseteq Z \setminus \{v,v'\}$ such that $v$ and $v'$ are disconnected in $G[Z \setminus C]$. If $v$ is irreflexive then $C \cup v$ is a stable cut, that is we may assume $v$ and symmetrically $v'$ have a self-loop. Let $w$ be the first irrefexive vertex in $P$. We obtain $G'$ by making every vertex in $P \setminus w$ reflexive and removing vertices $Z \setminus P$.

If there is no such set $C$, we consider the following cases: both $v,v'$ are reflexive, one is reflexive or both are irreflexive. In the first case as $v,v'$ must be in one partition of any stable cut meaning we may obtain $G'$ by making every vertex in $P$ reflexive and removing vertices $Z \setminus P$. In the second case without loss say $v$ is irreflexive. As there is no stable set $C$ as described above $Z \setminus v$ must be contained in one partition of any stable cut. We now obtain $G'$ by making every vertex in $P \setminus v$ reflexive and removing vertices $Z \setminus P$. Finally, if $v, v'$ are irreflexive, as $v,v'$ are not a stable cut they must be adjacent. No vertex of $Z \setminus \{v,v'\}$ is contained in some minimal stable cut so let $G' =G \setminus (Z \setminus \{v,v'\})$. Given $G[Z]$ contains at least one degree 3 vertex, in each case at least one vertex has been removed.
\end{itemize}
\qed
\end{proof}

\begin{theorem}\label{thm:tala-H2221}
{\sc Pseudograph Stable Cut} is polynomial-time solvable for $\{H_1^{2,2,2,1},$ $K_3\}$-subgraph-free graphs.
\end{theorem}

\begin{proof}
Let $G$ be some  $\{H_1^{2,2,2,1}, K_3\}$-subgraph-free graph. We first complete the following pre-processing steps outlined by Lemmas~\ref{lem:gen-obs} and \ref{lem:small-cut} with $k=10$.

We now claim $S$ is a stable set. Given $S$ contains all irreflexive vertices, $G$ has a pseudograph stable cut if, and only if $S$ is also a cutset. Assume for contradiction there is some pair of adjacent irreflexive vertices $p,q$ from Property $1$ of Lemma~\ref{lem:gen-obs} both have degree $\geq 3$. As $G$ is $K_3$-subgraph-free, there are distinct vertices $a,b \in N(p)$ and $c,d \in N(q)$. Let $D=\{p,q,a,b,c,d\}$.

From Lemma~\ref{lem:small-cut}, there must be three vertices, $x,x',x'' \in D$ with some neighbour in $G \setminus D$. Further we may assume these are distinct neighbours in $G \setminus D$, else again there are some pair of gateway vertices. Without loss we may assume $x = a$ given $a,b,c,d$ are symmetric.

\begin{figure}
    \centering
    \begin{tikzpicture}[scale=0.75, every node/.style={circle, draw, fill=white, inner sep=3pt}]
      \node (p) at (0,2) {$p$};
      \node (q) at (2,2) {$q$};
      \node (a) at (-1,0) {$a$};
      \node[draw=black, very thick] (a') at (-1,-2) {$a'$};
      \node (b) at (0,0) {$b$};
      \node (c) at (1,0) {$c$};
      \node[draw=black, very thick] (c') at (1,-2) {$c'$};
      \node (d) at (2,0) {$d$};
      \node (q') at (3,0) {$q'$};
    
      \draw (p) -- (q);
      \draw (p) -- (a);
      \draw (p) -- (b);
      \draw (q) -- (c);
      \draw (q) -- (d);
      
      \draw (a) -- (a');
      \draw (c) -- (c');
      \draw (q) -- (q');
      \draw[very thick] (d) -- (c');
    
      \draw[bend left=45, dotted] (c) to (a);
      \draw[dotted] (c) to (a');

      \node (p2) at (7,2) {$p$};
      \node (q2) at (9,2) {$q$};
      \node (a2) at (6,0) {$a$};
      \node[draw=black, very thick] (a'2) at (6,-2) {$a'$};
      \node (b2) at (7,0) {$b$};
      \node (c2) at (8,0) {$c$};
      \node[draw=black, very thick] (c'2) at (8,-2) {$c'$};
      \node (d2) at (9,0) {$d$};
      \node (q'2) at (10,0) {$q'$};
    
      \draw (p2) -- (q2);
      \draw (p2) -- (a2);
      \draw (p2) -- (b2);
      \draw (q2) -- (c2);
      \draw (q2) -- (d2);
      
      \draw (a2) -- (a'2);
      \draw (c2) -- (c'2);
      \draw (q2) -- (q'2);
    
      \draw[bend left=35] (d2) to (a2);
      \draw (q'2) to (a'2);
    \end{tikzpicture}

\caption{\textbf{Case 1}, $x=a',x'=c'$. Left: If the edge $(d,c')$ exists, at least one of the dotted edges must exist. Right: Given $dc'$ is not an edge, both the edges $ad$ and $a'q'$ must exist.}\label{fig:case1}
\end{figure}

In case 1 we assume $x'=c$, that is, there exist distinct vertices  $a' \in N(a) \setminus D$ and $c' \in N(c) \setminus D$. As $G$ is $H_1^{2,2,2,1}$-subgraph-free, $N(b) \subseteq D \cup \{a',c'\}$, symmetrically the same holds for $d$. Without loss we may assume $x'' =q$, that is there is some $q' \neq a' \neq c' \in N(q) \setminus D$. 

Further, $N(d) \subseteq \{q,a,a',c'\}$ else there is some $K_3$ or $H_1^{2,2,2,1}$ with degree 3 vertices $p,q$. Assume $d$ is adjacent to $c'$, see Figure~\ref{fig:case1} (left), now symmetrically $N(c) \subseteq \{q, a,a',c'\}$. As $N(c) \neq N(d)$ we may assume without loss $c$ is adjacent to either $a$ or $a'$. However both of these lead to a $H_1^{2,2,2,1}$ with degree $3$ vertices $c,q$. Therefore $N(d) \subseteq \{q, a,a'\}$ and symmetrically $N(q') \subseteq \{q, a,a'\}$. Both $d$ and $q'$ have degree at least $2$ and $N(d) \neq N(q')$ so without loss we may assume $a \in N(d)$ and $a' \in N(q')$, Figure~\ref{fig:case1} (right). This is a contradiction as there is some $H_1^{2,2,2,1}$ with degree 3 vertices $p,q$ concluding case 1.

\begin{figure}
    \centering
    \begin{tikzpicture}[scale=0.75, every node/.style={circle, draw, fill=white, inner sep=3pt}]
      \node (p) at (-1,2) {$p$};
      \node (q) at (1,2) {$q$};
      \node (a) at (-2,0) {$a$};
      \node (a')[draw=black, very thick] at (-2,-2) {$a'$};
      \node (b) at (-1,0) {$b$};
      \node (c) at (1,0) {$c$};
      \node (d) at (2,0) {$d$};
      \node (p')[draw=black, very thick] at (0,0) {$p'$};
    
      \draw (p) -- (q);
      \draw (p) -- (a);
      \draw (p) -- (b);
      \draw (q) -- (c);
      \draw (q) -- (d);
      
      \draw (a) -- (a');
      \draw (p) -- (p');
    
      \draw[bend left=40] (c) to (a);
      \draw (d) to (a');
    \end{tikzpicture}
\caption{\textbf{Case 2}, $x=a',x'=p'$. The edges $ac$ and $a'd$ must exist else this is case 1.}
\label{fig:case2}
\end{figure}

For case 2 we assume $x' = p$, that is there are distinct vertices $a' \in N(a) \setminus D$ and $p' \in N(p) \setminus D$. Now $N(c) \subseteq D \cup \{p',a'\}$ else this is case 1. Likewise $c$ is not adjacent to $p'$ (symmetrically $b$) else again this is case 1 relabelling $p'$ as $c'$. Given $G$ is $K_3$-subgraph-free, $N(c) \subseteq \{q,a,a'\}$ and symmetrically $N(d) \subseteq \{q,a,a'\}$. As both $c$ and $d$ have degree at least $2$ and $N(c) \neq N(d)$ without loss $a \in N(c)$ and $a' \in N(d)$, given $c$ and $d$ are not necessarily incomparable, we do assume $a \notin N(d)$. See Figure~\ref{fig:case2}. We now consider the possible $x''$ vertices, claiming each result in something isomorphic to case 1. If $x'' =q$, there is some $q' \neq p' \neq q' \in N(q) \setminus D$, which is case 1 given the following relabelling: $$l_1 : (p,q,a,b,c,d,a',p',q') \to (p,q,a,b,a',c,c',p',d)$$ If $x''=b$, then $b' \neq p' \neq q' \in N(b) \setminus D$, again this is case 1 given the following relabelling: $$l_2 : (p,q,a,b,c,d,a',p',b') \to (p,q,b,a,d,c,c',p',a')$$ This concludes case 2.

\begin{figure}
    \centering
    \begin{tikzpicture}[scale=0.75, every node/.style={circle, draw, fill=white, inner sep=3pt}]
      \node (p) at (0,2) {$p$};
      \node (q) at (2,2) {$q$};
      \node (a) at (-1,0) {$a$};
      \node (a')[draw=black, very thick] at (-1,-2) {$a'$};
      \node (b) at (0,0) {$b$};
      \node (b')[draw=black, very thick] at (0,-2) {$b'$};
      \node (c) at (1,0) {$c$};
      \node (d) at (2,0) {$d$};
      \node (q')[draw=black, very thick] at (3,0) {$q'$};
    
      \draw (p) -- (q);
      \draw (p) -- (a);
      \draw (p) -- (b);
      \draw (q) -- (c);
      \draw (q) -- (d);
      
      \draw (a) -- (a');
      \draw (b) -- (b');
      \draw (q) -- (q');
    \end{tikzpicture}
    
    \caption{\textbf{Case 3}, $x=a',x'=b', x''=q'$}
    \label{fig:case3}
\end{figure}

This leaves only case 3 where $x'=b$ and $x'' = q$, that is there are distinct vertices $a' \in N(a) \setminus D$ and $b' \in N(b) \setminus D$ and  $q' \in N(q) \setminus D$, Figure~\ref{fig:case3}. As previously $N(c) \subseteq D \cup \{a',b',q'\}$ else this is case 1. Further if $c$ is adjacent to $b'$, then relabelling $b'$ to $c'$ shows again we are in case 1. Symmetrically, $c$ is not adjacent to $a'$ and given $G$ is $K_3$-subgraph-free, $N(c) \subseteq \{q,a,b\}$. As $N(c) \subseteq N(p)$, $c$ must be reflexive. Notice $c,d,q'$ are symmetric cases so all must be reflexive and $N(\{c,d,q'\}) \subseteq \{q,a,b\}$. Given every vertex has degree at least $2$ there must be some pair $v \neq v' \in \{c,d,q'\}$ , such that $N(v) \subseteq N(v')$. However as $c,d,q'$ are each reflexive this is a contradiction. \qed
\end{proof}
This yields the polynomial time proof of the 7th row of Table~\ref{tab:1}. Together with Corollary~\ref{cor:main-stable-bis}, the previous theorem yields another proof of Theorem~\ref{t-3}.

\section{The Proof of Theorem~\ref{thm:C'r-triangle-free}.}\label{a-first-3}

Here is the missing proof of Theorem~\ref{thm:C'r-triangle-free}.

\medskip
\noindent 
\textbf{Theorem~\ref{thm:C'r-triangle-free}.}
{\sc Pseudograph SurHom}($C'_r$) is \NP-complete restricted to triangle-free pseudographs for $r \geq 6$.

\begin{proof}
Let $r\geq 6$. We reduce from {\sc Conflict-Free Cut} which is \NP-complete even when the conflict graph $\hat{G}$ is $1$-regular~\cite{RAUCH2025106503}, that is, even when $\hat{G}$ is an induced perfect matching. 

Given an instance $(G, \hat{G})$ of {\sc Conflict-Free Cut}, we construct an instance $G'$ of {\sc Pseudograph} \RSurHom($C'_r)$ where $C'_r$ has vertex set $\{0,1,2, \dots r-1 \}$ and $0$ is the only irreflexive vertex. We begin with a (irreflexive) copy of $\hat{G}$ with vertex set $M=E(G)$.\footnote{These edges we will write in concatenation without brackets and comma.} Next we take a reflexive independent set $N_v= \{v'_1, \dots v'_{deg(v)}\}$, and a further reflexive vertex $v''$ which is complete to $N_v$, for each vertex $v$ of $G$. We then add an irreflexive vertex $d$ which is complete to $N=\bigcup_{v \in V(G)}N_v$. We add a reflexive pendant path on  $s=\lfloor \frac{r}{2} \rfloor-2$  vertices, $v''_{i_1} \dots v''_{i_s}$, attaching one end to $v''$, for each vertex $v$ of $G$. For each edge, $e=uv$ of $G$, we choose a unique vertex $u'_i$ of $N_u$ and add the edge $eu'_i$.  We also choose a unique vertex $v'_j$ of $N_v$ and add the edge~$ev'_j$.
We refer to Figure~\ref{fig:surhom} for an illustration where $r=7$. In the same figure, we also illustrate the case $r=6$, which required us to make a slight modification of the general case, as we explain below.

We claim that $(G,\hat{G})$ has a conflict-free cut if and only if $G'$ permits a surjective homomorphism to $C'_r$. First assume $(G, \hat{G})$ has a conflict-free cut $F$, which separates the vertices of $G$ into two sides A and B. We obtain a surjective homomorphism $h$ from $G'$ to $C'_r$: let $h(d)=0$; for every $i$ and $v\in A$, let $h(v'_i)=1$, and
for every $i$ and $v\in B$, let $h(v'_i)=r-1$.

We now consider the (isolated) edges $ef$ in $\hat{G}$, so $e$ and $f$ are edges in $G$. As $F$ is a conflict-free cut of $(G,\hat{G})$, at most one of $e,f$ belongs to $F$. Say $e\in F$, so $f\notin F$. We can safely set $h(e)=0$ and $h(f)=1$ or $h(f)=r-1$, depending on whether the two end-points of $f$ belong to $A$ or $B$, respectively. Now suppose neither $e$ nor $f$ belongs to $F$. If both $e$ and $f$ have their two endpoints both in $A$, or both in $B$, then we let $h(e)=h(f)=1$, or $h(e)=h(f)=r-1$, respectively. If one of them, say $e$, has its two endpoints in $A$, while $f$ has its two endpoints in $B$, then we can safely set $h(e)=1$ and $h(f)=0$. Finally, beginning with $v''$, we map the pendant paths either to $2,3, \dots , \lfloor {\frac{r}{2}} \rfloor$ (if $h(v'_i)=1$) or $(r-2),(r-3), \dots ,(r- \lfloor{\frac{r}{2} \rfloor})$ (if $h(v'_i)=r-1)$. This ensures the surjectivity of~$h$. 

Now suppose $G'$ has a surjective homomorphism $h$ to $C'_r$. We first claim that $h(d)=0$. Otherwise, since every irreflexive vertex shares a neighbour with $d$,  we may assume without loss of generality that $h(d) \in \{1,2 \}.$ For $r \geq 7$, this contradicts the surjectivity of $h$ since $d$ is joined to every other vertex of $G'$ by a path of length at most $\lfloor \frac{r}{2} \rfloor$ on which every internal vertex is reflexive.

As mentioned, for $r=6$, we slightly modify the construction by identifying the degree $1$ endpoints of each pendant path in $G'$ (which is a vertex), to a single vertex $x$. Note that we still obtain a surjective homomorphism from a conflict-free cut of $(G, \hat{G})$ as described above, since each of these identified vertices is mapped to $3$.  Now, we argue that $G'$ has diameter at most $4$. Any two vertices at a distance of at most $2$ from $d$ are at a distance of at most $4$. The only vertex $x$ with $dist(d,x)>2$ is the vertex obtained by identifying the endpoints of the reflexive paths. In this case $dist(x,v')=2$ for every $v' \in N(d)$. Hence $dist(x,y) \leq3$ for every vertex $y$ of $G'$. Now we may assume that $G'$ has no surjective homomorphism to $P'_6$, a path of length $6$ where every vertex except for one endpoint is irreflexive. In particular, this implies that, in any surjective homomorphism from $G'$ to $C_6'$, there exists a vertex  $v \in h^{-1}(0)$ which has neighbours in both $h^{-1}(1)$ and $h^{-1}(5)$. If $h(d) \in \{1,2\}$ then, since every neighbour of $d$ is reflexive, there is no vertex $y \in N^2(d)$ with $h(y)=5$. This contradicts the fact that $h$ is a surjective homomorphism, as the neighbourhood of any vertex $z \in h^{-1}(0)$ is contained in~$N^2(d)$. We conclude that $h(d)=0$ both if $r=6$ and $r\geq 7$.

Therefore, since $h(d)=0$ for any surjective homomorphism $h$ from $G'$ to $C'_r$, $h(v'_i) \in \{1, r-1 \}$ for each vertex $v'_i$ of $N$. We obtain a conflict-free cut of $G$ consisting of all the edges $xy$ such that $h(x'_1)\neq h(y'_1)$. By surjectivity of $h$ and the fact that each vertex $u$ of $G'$ is joined to $d$ by a path of length at most $\lfloor {\frac{r}{2}} \rfloor$ with only reflexive internal vertices, at least one vertex of $N$ is mapped to each of $1$ and $r-1$. Therefore both $A= \{v: h(v'_1)=1\}$ and $B= \{v: h(v'_1)=r-1 \}$ are non-empty. Any two vertices  $v'_i$,  $v'_j \in N(d)$ corresponding to the same vertex $v$ of $G$ are mapped to the same vertex since they share a reflexive neighbour and a neighbour $d \in h^{-1}(0)$. Therefore, if an edge has endpoints in both $A$ and $B$, its corresponding vertex in $M$ has neighbours mapped to both $r-1$ and $1$. Hence, this vertex must map to $0$. Then the set of edges with one endpoint in $A$ and the other in $B$ is contained in $h^{-1}(0)$. In particular, it induces an independent set and thus is a conflict-free cut of~$(G,\hat{G})$. 

We claim that $G'$ is triangle-free. Since $N$ is an independent set, neither $d$ nor any reflexive neighbour of $v'_{i}$ is contained in a triangle. Similarly, no pendant path vertex is contained in any triangle. Then any triangle contains two adjacent vertices of $M$.  These two vertices do not share a neighbour in $M$ since $\hat{G}$ is $1$-regular. They do not share a neighbour in $N$ since each vertex of $N$ has at most one neighbour in $M$. Therefore $G'$ is triangle-free. \qed
\end{proof}   

\begin{figure}[h]
\centering \includegraphics[width=0.8\linewidth,page=3]{IPE-figs-multi.pdf}
 \caption{The part of $G'$ representing an edge $uv-vw$ in $\hat{G}$ if $r=6$ (left) and if $r=7$ (right).}
 \label{fig:surhom}
 \end{figure}
 
\end{document}